\documentclass{article}

\usepackage{amsmath,amsxtra,amssymb,amsfonts,latexsym, mathrsfs, dsfont,bm,enumitem, amsthm, mathtools}
\usepackage{bbm}
\usepackage[usenames]{color}
\usepackage[hyperindex]{hyperref}
\usepackage[initials]{amsrefs}
\usepackage{tikz}
\usepackage{comment}
\usepackage{xcolor}
\usepackage{graphics,epsf} 
\usepackage{graphicx}
\usepackage{wrapfig}        %   OK for MAC
\usepackage{caption}
\newtheorem{theorem}{Theorem}[section]
\newtheorem{proposition}[theorem]{Proposition}

\newtheorem{corollary}[theorem]{Corollary}
\newtheorem{lemma}[theorem]{Lemma}

\theoremstyle{definition}

\newcommand{\cn}{\mathcal N}

\newcommand{\N}{\mathbb{N}}
\newcommand{\Z}{\mathbb{Z}}

\newcommand{\R}{\mathbb{R}}

\newcommand{\varep}{\varepsilon}

\newcommand{\bee}{\mathbf{e}}

\newcommand{\bz}{\mathbf{0}}
\newcommand{\bone}{\mathbf{1}}

\newcommand{\bsb}{\boldsymbol}

\def\one{{\mathbf 1}}
\title{Oscillatory Loomis-Whitney and Projections of Sublevel Sets}

	\newcommand{\Addresses}{{
		\bigskip
		\footnotesize
		
		\textsc{Department of Mathematics, Michigan State University, East Lansing, Michigan 48824}\par\nopagebreak
		\textit{E-mail address}: \texttt{gilulama@math.msu.edu}
		
		\medskip
		
		\textsc{Department of Mathematics, UC Berkeley, Berkeley, California 94720}\par\nopagebreak
		\textit{E-mail address}: \texttt{oneill@math.berkeley.edu}
		
		\medskip
		
		\textsc{Google Brain, Mountain View, California 94043}\par\nopagebreak
		\textit{E-mail address}: \texttt{lechao.xiao@gmail.com}}}

\begin{document}
	\author{Maxim Gilula, Kevin O'Neill, and Lechao Xiao}
	\maketitle
	\begin{abstract}
		We consider an oscillatory integral operator with Loomis-Whitney multilinear form. The phase is real analytic in a neighborhood of the origin in $\R^d$ and satisfies a nondegeneracy condition related to its Newton polyhedron. Maximal decay is obtained for this operator in certain cases, depending on the Newton polyhedron of the phase and the given Lebesgue exponents. Our estimates imply volumes of sublevel sets of such real analytic functions are small relative to the product of areas of projections onto coordinate hyperplanes. 
	\end{abstract}

	\section{Introduction}
	In this paper we study multilinear oscillatory integral operators of the form
	
	\begin{equation}\label{eq: our op}
		\Lambda_d(f_1,\dots,f_d)= \int_{\R^d} e^{i\lambda S(x)}\psi(x) \prod_{j=1}^d f_j(\pi_j(x))  dx,
	\end{equation}
	where $\lambda\in \R$ is a real parameter, $S(x)$ is real-analytic in a neighborhood of the origin, $\psi\in C_0^\infty(\R^d)$, and  $\pi_j:\R^d\to \R^{d-1}$ are the surjective linear maps defined by 
	\begin{equation*}
		\pi_j(x)=(x_1,\dots,x_{j-1},x_{j+1},\dots, x_d):=\hat{x}_j.
	\end{equation*}
	
	A fundamental question answered positively in \cite{CLTT} is the following: Given exponents $p_1,\dots,p_d$, do there exist $C>0$ and $\epsilon>0$ such that 
	\begin{equation*}
		|\Lambda_d(f_1,\dots,f_d)|\leq C|\lambda|^{-\epsilon}\prod_{j=1}^d\|f_j\|_{p_j}
	\end{equation*}
	uniformly in the $f_j$?
	
	The case $d=2$ has been very well studied under various nondegeneracy conditions (e.g., \cites{ccw99, cw02, Hor, ps94}), so we only focus on $d\ge 3$. In this case, Christ-Li-Tao-Thiele \cite{CLTT} showed existence of decay for nondegenerate phases for the full admissible range of exponents $p_j$. However, their research provided an existence statement, leaving the sharp decay rate a mystery. The focus of our research herein is to provide an explicit decay rate, which in many cases turns out to be the sharp decay rate, depending geometrically on the Newton polyhedron of $S$ and the exponents $(p_1,\dots, p_d).$ The phases $S$ considered here are closely related to those studied by Varchenko\cite{varchenko76} in the scalar case, and by Gilula-Gressman-Xiao in the multilinear case\cite{GiGrXi18}.
	
	In the $d=2$ case, H\"ormander\cite{Hor} showed $\lambda^{-1/2}$ decay for nondegenerate phases when $p_1=p_2=2$. This is sharp in the $L^2$ case and even the $L^\infty$ case, as may be shown by example. In certain cases, our estimates contain a decay rate of $\lambda^{-2^{1-d}}$ and are sharp for particular choices of $p_j$, though the picture is unclear when $p_j\equiv\infty$. In the $L^\infty$ case, we are aware only of examples bounding the decay rate as at most $\lambda^{-1/d}$.
	
	We will provide a heuristic that the best expected decay is $\lambda^{-1/\delta}$, where $\delta$ is tied to both the $p_j$ and to the boundary of the Newton polyhedron $\cn(S)$ of $S$. This $\delta$ is a generalization of the so-called Newton distance, which is the sharp exponent of decay for scalar oscillatory integrals, assuming $S$ satisfies a certain nondegeneracy condition. (See Varchenko's work in \cite{varchenko76}.)
	
	Similar to Varchenko, our setup begins with a real analytic phase $S:\R^d\to\R$ satisfying a nondegeneracy condition depending on its Newton polyhedron, and $\psi:\R^d\to\R$ is a smooth cutoff function supported in a small enough neighborhood of the origin, depending on $S.$ (Formal definitions of nondegeneracy and of the Newton polyhedron will be given in Section \ref{sec: background}.) Our estimates are not uniform in any easily describable class of phases, analogous to the estimates in \cite{GiGrXi18}; if one is willing to trade a larger range of exponents for which a sharp estimate holds, a uniform estimate could potentially be obtained by an approach similar to that of Phong-Stein-Sturm\cite{pss01}.
	
	We provide two results in our theorem: an on-diagonal version and an off-diagonal version. There may be many different ways to interpolate between the conclusions of these theorems, so we simply leave this theorem with two possible hypotheses. At the end of Section \ref{sec: main proof}, we discuss other possible interpolations one could do with our global and local estimates, and provide an example illustrating one of the many nontrivial ways to interpolate our results: the failure of obtaining a sharp decay rate by interpolate global estimates can be ameliorated by first applying local estimates proved below, and then Corollary \ref{cor: linop}. The main issues with local Loomis-Whitney inequalities that prevent us from easily interpolating all possible cases is that the constants in these local inequalities depend on the domain.
	
	For exponents $p_1,\dots, p_d,$ we write $P=\sum_{j=1}^dp_j^{-1}$, following the convention that $\infty^{-1}=0$. Let $p=p(d)$ be defined by $p(d)=(d-1)\frac{2^{d-1}}{2^{d-1}-1}$.
	
	\begin{theorem}[Oscillatory Loomis-Whitney]\label{thm: main theorem}
		Let $S:\R^d\to\R$ be a nondegenerate real analytic function for $d\ge 3$. Let $\psi:\R^d\to\R$ be smooth and supported in a small enough neighborhood of the origin. Let $f_j\in L^{p_j}(\R^{d-1})$, where
		\begin{itemize}
			\item[(i)](Off-diagonal) $p_1\ge 2^{d-1},p_2\ge 2^{d-1}, p_3\ge 2^{d-2},\cdots,p_d\geq2$; or
			\item[(ii)](On-diagonal) all $p_j\ge p(d).$
		\end{itemize} 
		Let $\delta>0$ be such that 
		\begin{equation*}
			\delta(1-P+p_1^{-1},\dots, 1-P+p_d^{-1})\in\partial\cn(S).
		\end{equation*} 
		For all $\lambda\in\R\setminus\{0\},$
		\begin{equation*}
			|\Lambda_d(f_1,\dots,f_d)|\lesssim \prod_{j=1}^d \|f_j\|_{p_j}
			\begin{cases}
				|\lambda|^{-1/\delta}\log^{d}(2+|\lambda|)  & \text{ if } \delta \ge 2^{d-1},\\
				|\lambda|^{-1/2^{d-1}} & \text{ if }\delta < 2^{d-1},
			\end{cases}
		\end{equation*}
		where the implicit constant is independent of $f_1,\dots, f_d$ and $\lambda$.
	\end{theorem}

    The reader should consult Corollary \ref{cor: linop} to see that a sharper result holds in in the case $\delta>2^{d-1}$ (strictly smaller power of $\log,$ depending on the Newton polyhedron of $S$). There is an additional subtlety when $\delta=2^{d-1}$ that appeared even in \cite{pss01}, where like Phong-Stein-Sturm, we also require more than the expected $d-1$ many factors of $\log(2+|\lambda|)$.
    
	Note that in the off-diagonal version, $P$ is always less than or equal to $1,$ but in the on-diagonal version, $P$ can be $d\frac{2^{d-1}-1}{(d-1)2^{d-1}}>1,$ and therefore we cannot simply interpolate the results under hypothesis $(i)$ to obtain the results of the theorem under hypothesis $(ii).$ 
	
	It is worth noting that indeed $\delta<\infty$ exists, as the reader may check that for all $j$, the quantity $1-P+p_j^{-1}$ is positive (and nondegeneracy implies $\delta>0$).
	
	As the classical Loomis-Whitney inequality states that $\int_{\R^d}\prod_{j=1}^df_j(\hat{x}_j)dx\leq\prod_{j=1}^d\|f_j\|_{q_j}$ with $q_j\equiv d-1$, one might be surprised to see the exponent 2 in our off-diagonal estimates when $d\geq4$. However, this inequality holds for a much greater range of exponents $q_j$ when the domain is compact, e.g., (8,4,4,2). (See \cite{BCCT} for a general theory describing this phenomenon as well as the relation between exponents for oscillatory and non-oscillatory multilinear integrals.)

	As a standard application of our main theorems, we obtain the sublevel set estimate of Corollary \ref{cor: sublevel} by applying Lemma \ref{lem: sublevel}.

    \begin{lemma}\label{lem: sublevel}
	Let $S:\mathbb{R}^d\to\mathbb{R}$ be real analytic in a neighborhood of the origin, let $\pi_j:\mathbb{R}^d\to\mathbb{R}^{d_j}$ be surjective linear maps, and let $B$ be a sufficiently small neighborhood of the origin.

Suppose there exists a constant $C<\infty$ and $M(\lambda)$ such that for all $f_j\in L^{p_j}(\mathbb{R}^{d_j})$,

\begin{equation*}
    \Bigg{|}\int_B e^{i\lambda S(x)}\prod_{j=1}^n(f_j\circ \pi_j)(x) dx \Bigg{|}\le C M(\lambda)\prod_{j=1}^n\|f_j\|_{p_j}.
\end{equation*}

% Let $S,\psi,f_1,\dots,f_n:\mathbb{R}^{k_j}\to\mathbb{R}$ be locally integrable functions. Assume $\psi$ is supported in $[-1,1]^d.$ Let $L_j:\mathbb{R}^{d}\to\mathbb{R}^{k_j}$ be surjective linear maps. Assume there is a nonnegative constant $M$ such that 
% \begin{equation*}
%     \Bigg{|}\int_{\mathbb{R}^d}e^{i\lambda S(x)}\psi(x)\prod_{j=1}^n(f_j\circ L_j)(x) dx \Bigg{|}\lesssim M(\lambda)\prod_{j=1}^n\|f_j\circ L_j\|_{p_j},
% \end{equation*}
% where the implicit constant is independent of $\lambda, f_j, L_j$, and $M$ satisfies 

Furthermore, suppose $M(\lambda)$ satisfies

\begin{itemize}
\item[(i)] $M(\lambda)\le (1+|\lambda|)^{-\delta}$ for some $\delta>0$, and
\item[(ii)] $M(\lambda_1\lambda_2)\le A M(\lambda_1) M(\lambda_2)$ for some $A>0$ independent of $\lambda.$
\end{itemize}
% \begin{itemize}
% \item[(i)] there is a $\delta>0$ such that for all $\lambda\neq 0,$ $M(\lambda)\le |\lambda|^{-1+\delta}$, and
% \item[(ii)] $M(\lambda_1\lambda_2)\le M(\lambda_1) M(\lambda_2).$
% \end{itemize}

Then, there exists $C'$ independent of $\varepsilon$ such that for all measurable functions $g_j:\mathbb{R}^{d_j}\to\mathbb{R}$,

\begin{equation*}
    |\{x\in B:|S(x)-\sum_{j=1}^n(g_j\circ \pi_j)(x)|<\epsilon\}|\leq C'M(\epsilon^{-1})
\end{equation*}
\end{lemma}

Typically, one takes $M(\lambda)=\lambda^{-\alpha}$ for some $\alpha>0$, though we allow for greater generality to handle the factors of $\log\lambda$ coming from Theorem \ref{thm: main theorem}.
	
	\begin{corollary}\label{cor: sublevel}
		Let $d\geq 1$ and let $B\subset \R^d$ be a bounded set which is sufficiently close to the origin in the sense of Theorem \ref{thm: main theorem}. Suppose $S(x)$ is a phase satisfying the hypotheses of Theorem \ref{thm: main theorem}. Then there exists $C>0$ such that for all functions $g_j:\R^{d-1}\rightarrow\R$,
		\begin{equation}\label{eq: sublevel cor}
			|\{x\in B:|S(x)-\sum_jg_j(\pi_j(x))|<\epsilon\}|\leq C\begin{cases}
				\epsilon^{1/\delta}\log^{d}(\epsilon^{-1})  & \text{ if } \delta \ge 2^{d-1},\\
				\epsilon^{1/2^{d-1}} & \text{ if }\delta < 2^{d-1},
			\end{cases}
		\end{equation}
		where $\delta(1,\dots,1)\in\partial\cn(S)$.
	\end{corollary}
    Lemma \ref{lem: sublevel} was implicitly proven in section 7 of Christ-Li-Tao-Thiele\cite{CLTT}.
	One could also compare Corollary \ref{cor: sublevel} to the sublevel set estimates in Greenblatt\cite{gr10}, or the one in \cite{pss01}, where Phong-Stein-Sturm consider perturbations of the form $g(x_j)$ for polynomial $S.$
	
	Nondegenerate phases are, in a sense, generic. Thus, Corollary \ref{cor: sublevel} may be interpreted as saying that generically, the product of areas of projections onto coordinate hyperplanes of a real analytic function has very large area compared to the volume of the sublevel set of the function. To see this, take $g_j\equiv0$ so the left hand side of \eqref{eq: sublevel cor} becomes the volume of the sublevel set, and observe that our nondegenerate phases are 0 on coordinate hyperplanes. Compare this with the Loomis-Whitney inequality: characteristic functions of boxes are maximizers of Loomis-Whitney, and the Loomis-Whitney inequality compares sizes of projections to the volume. Theorem \ref{thm: main theorem} provides extra decay in the inequality, and therefore quantifies in a certain sense how far boxes are from varieties of those real analytic functions perturbed by locally measurable $g(\hat{x}_j)$.
	
	In Section \ref{sec: background}, we formally define the Newton polyhedron and our notion of nondegeneracy. Using these principles, we describe some intuition for the statement of Theorem \ref{thm: main theorem} and show it is sharp when $\delta\geq2^{d-1}$. We also take this as an opportunity to prove growth estimates which follow from our definition of nondegeneracy and will be useful later on. The lower bound on $S$ in small regions is one of the more subtle arguments required to make the theorem work in its full generality.
	
	We split the proof of our local estimates into Sections \ref{sec: d=3} for the $d=3$ case and \ref{sec: higher d} for the $d>3$ case. The reason is that in inducting on dimension, there is a single $d=2$ result to refer to; however, the resulting theorem in $d=3$ naturally features mixed norms. This result implies one with equal exponents (on-diagonal) and unequal exponents (off-diagonal), neither of which implies the other. Once we have made this split, there are two separate theorems in each dimension to induct on. Much care is required with these interpolations and applications of basic inequalities, in particular with the management of mixed norms; there are many choices to make, and if any were made in the wrong order, we would not be able to arrive at our full range of exponents in Theorem \ref{thm: main theorem}.
	
	In Section \ref{sec: main proof}, we decompose our domain into dyadic boxes, applying our local estimates on each one and optimizing over the sum, as in \cite{GiGrXi18}. %Section \ref{sec: sublevel} contains the proof of Corollary \ref{cor: sublevel}, which is an interesting result in its own right since it sheds light on real-analytic varieties.
	
	\subsection{Notation and conventions}
	Here we list a small but important index of notation used throughout.
	\begin{itemize}
		\item We include zero in the natural numbers: $\N=\{0,1,2,\dots\}.$
		\item If $c$ is a constant, we denote by bold $\boldsymbol{c}$ the vector $(c,\dots,c)\in \R^d.$ 
		\item In general, if $\alpha\in\N^d,$ the notation $\partial^\alpha=\partial_{x_1}^{\alpha_1}\cdots\partial_{x_d}^{\alpha_d}=\partial_1^{\alpha_1}\cdots\partial_d^{\alpha_d}$ is common. Due to ambiguities resulting from $\partial^\one$, we additionally define the operator $D_d=\partial_1\cdots\partial_d$.
		\item Due to symmetry of the scalar $\lambda$ appearing in $\Lambda_d$, we restrict our attention to $\lambda>2$, avoiding writing $|\lambda|$ each time.
		\item If $x\in\R^n$ and $y\in\R^m$, and $f:\R^n\times\R^m\to \R$ is such that the following integral is well-defined, we denote the mixed $L^p_x(L^q_y)$ norm of $f$ by
		\begin{equation*}
			\|f\|_{L^p_x L^q_y}=\Big{\|}\|f(x,\cdot)\|_{L^q}\Big{\|}_{L^p} = \left(\int_{\R^n}\left(\int_{\R^m} |f(x,y)|^q dy\right)^{p/q}dx\right)^{1/p}.
		\end{equation*}
		\item Whenever exponents $p_1,\dots,p_d$ are clear from context, we define $P=\sum_{j=1}^d p_j^{-1}.$
		\item Given a phase $S(x)$, and exponents $(p_1,\dots, p_d)$, let $\delta$ be such that 
		\begin{equation*}
			\delta(1-P+p_1^{-1},\dots, 1-P+p_d^{-1})\in\partial \cn(\phi).
		\end{equation*} 
		We define $\delta$ to be \textit{the Newton distance of $S$ with respect to} $(p_1,\dots, p_d)$. We call $\delta$ \textit{the Newton distance with respect to} $(p_1,\dots, p_d)$, or briefly \textit{the Newton distance}, when the context is clear.
	\end{itemize}

	\textbf{Acknowledgments}: The authors would like to thank Michael Christ, Philip Gressman, Ilya Kachkovskiy, and Willie Wong for many enlightening discussions.

	\section{Background}\label{sec: background}
	
	\subsection{Nondegeneracy and the Newton polyhedron}
	Let $S:\R^d\to\R$ be real analytic in a neighborhood containing the origin. Then, in a small box containing the origin, $S$ can be expressed as a power series $\sum_{\alpha\in(\N\setminus\{0\})^d}c_\alpha x^\alpha + \tilde{S}(x)$, where $\tilde{S}(x)$ may be written in the form $\sum_{j=1}^d g_j(\hat{x}_j)$ for functions $g_j$. The $\tilde{S}(x)$ part of the phase may be absorbed into the functions $f_j$; thus, it does not alter the decay rate in $\lambda$ corresponding to $S(x)$, so from here on we assume $\tilde{S}= 0$ (see the nondegeneracy condition of \cite{CLTT} for more details).
	
	The Newton polyhedron of $S$ is defined to be the convex hull of
	\begin{equation*}
		\bigcup_{c_\alpha\neq 0} \alpha+\R_+^d.
	\end{equation*}
	% 	It is a fact that the Newton polyhedron is indeed an unbounded convex polyhedron with finitely many compact faces, and therefore the union above can be expressed as a finite union. If $S(\bz)=0$, the polyhedron has some positive distance away from the origin. DO WE NEED ANY OF THIS?
	For compact faces $F\subset\cn(S),$ we define the polynomial $S_F(x)=\sum_{\alpha\in F} c_\alpha x^\alpha$. In this paper, we say that a real analytic function $S:\R^d\to\R$ is nondegenerate if for all $x$ not contained in coordinate hyperplanes, for all compact faces $F\subset\cn(S),$ the polynomial $S_F$ satisfies $D_dS_F(x)\neq 0.$ This nonvanishing ensures that, in some sense, we avoid cancellation and the phases behaves like its leading terms. Alternatively, one could hypothesize the growth condition found in Lemma \ref{lemma: lb}, but in practice the above condition is easier to verify.
	
	There is another nondegeneracy condition analogous to the one in Phong-Stein-Sturm \cite{pss01} that might sacrifice the range of exponents for uniform estimates: one could define a reduced Newton polyhedron from vertices $\alpha\in(\N\setminus\{0\})^d$ such that $|\partial^\alpha S(x)|\ge 1.$  One such example is building a polyhedron from the single vertex $\alpha=(1,\dots,1)$ which is equivalent to the condition $|D_dS(x)|\ge 1.$

	\subsection{Intuition: How all such multilinear estimates should depend on \texorpdfstring{$\cn(S)$}{N(S)}}
	
	Varchenko showed that for $S$ satisfying an analogous nondegeneracy condition,
	\begin{equation*}
		\int_{\R^d} e^{i\lambda S(x)}\psi(x)dx\lesssim \lambda^{-1/\delta}\log^{k-1}(\lambda),
	\end{equation*}
	as $|\lambda|\to\infty,$  where $\delta$ is the smallest real number such that $\boldsymbol{\delta}$ lies in $\cn(S)$, and $k\le d$ is the greatest codimension over all faces $F\subset \cn(S)$ containing $\boldsymbol{\delta}.$ Varchenko also proved both $k$ and the decay rate $\lambda^{-1/\delta}$ are the best possible if $\delta>1.$
	
	The first author showed in his thesis\cite{Gi18} that for any $\lambda>2$, $\beta\in(-1,\infty)^d$, and $S(x)$ satisfying Varchenko's condition,
	\begin{equation}\label{eq: Max's thesis}
		\int_{\R^d} e^{i\lambda S(x)}x^{\beta}\psi(x)dx\lesssim \lambda^{-1/\delta_\beta}\log^{k-1}(\lambda),
	\end{equation}
	where $\delta_\beta$ is the smallest such that $\delta_\beta(\beta+\bone)\in \cn(S)$ and $k\le d$ is the greatest codimension over all faces $F\subset \cn(S)$ containing $\delta_\beta(\beta+\bone).$ %Using a change of variables $x\mapsto y$ satisfying $dy_i=x_i^{\beta_i}dx_i$, and after modifying the definition of the Newton polyhedron to include more general lattices in $\R^d$, one might be able to modify Varchenko's argument to show this estimate is also sharp up to logs under some constraint on $\delta.$ 
	
	Using this scalar estimate, we now heuristically derive an upper bound for the decay rate of $\Lambda_d$. (By an upper bound for the decay rate, we mean the operator norm cannot decay any faster.) Using the scalar estimate above is a good heuristic because some phases satisfying Varchenko's nondegeneracy condition also satisfy ours.  Let $p_1,\dots,p_d$ satisfy $1-P+p_j^{-1}>0$ for $1\le j\le d.$ Let $\gamma>0$ be a small constant and define functions $f_j$ by
	\begin{equation*}
		(x_1\cdots x_{j-1}x_{j+1}\cdots x_d)^{-1/p_j+\gamma/(d-1)}=f_j(\hat{x}_j)\in L^{p_j}(\R^{d-1}).
	\end{equation*}
	Taking the product of the $f_j$, each $x_j$ receives exponents $\gamma/(d-1)-p_k^{-1}$ for $k\neq j,$ so the product equals
	\begin{equation*}
		\prod_{j=1}^d f_j(\hat{x}_j)=\prod_{j=1}^d (x_1\cdots x_{j-1}x_{j+1}\cdots x_d)^{-1/p_j+\gamma/(d-1)}=\prod_{j=1}^dx_j^{\gamma-P+1/p_j}.
	\end{equation*}
	Let $\beta$ be the vector defined componentwise by $\beta_j=(\gamma-P+p_j^{-1}).$ Then, by \eqref{eq: Max's thesis}, we expect
	\begin{equation*}
		\int_{\R^d} e^{i\lambda S(x)}\prod_{j=1}^d f_j(\hat{x}_j)\psi(x)dx\lesssim \lambda^{-1/\delta_\gamma}\log^{k-1}(\lambda)\lesssim \lambda^{-1/\delta_\gamma}\log^{k-1}(\lambda) \prod_{j=1}^d \|f_j\|_{p_j}
	\end{equation*}
	where $\delta_\gamma(1-P+p_1^{-1}+\gamma,\dots, 1-P+p_d^{-1}+\gamma)\in\partial\cn(S).$ Letting $\gamma\to 0,$ we expect the best exponent of $\lambda$ to be $-1/\delta$, where $\delta(1-P+p_1^{-1},\dots, 1-P+p_d^{-1})\in\partial\cn(S)$. Whenever $\delta\ge 2^{d-1},$ this is precisely the estimate we prove, and in this case we also prove this decay is maximal for $\Lambda_d$.  (In general our nondegeneracy condition does not imply that of Varchenko, hence this derivation is indeed heuristic.) Although it is very involved to examine when exactly this estimate is sharp in the scalar case, in our multilinear Loomis-Whitney scenario we can prove that the sharp exponent of $\lambda$ must be precisely what this heuristic predicts. We prove this claim in the next subsection. This provides evidence that the approach above may be useful for understanding the decay rates of a wide range of oscillatory integral operators.

	\subsection{Sharpness of Theorem \ref{thm: main theorem}}
	
	Since $\cn(S)$ is a polyhedron in $\R^d$ not containing $\bz$ (by nondegeneracy), its supporting hyperplanes may be defined by 
	\begin{equation*}
		H_\mathbf{n}=\{\alpha\in\R^d: \alpha\cdot \mathbf{n} = 1\},
	\end{equation*}
	where $\cdot$ is the usual inner product on $\R^d$. In particular, if $\alpha\in \partial\cn(S),$ any codimension 1 face containing $\alpha$ is a subset of the supporting hyperplane $H_{\bsb{n}}$ where $\alpha\cdot\bsb{n}=1.$ Let $\delta$ be the Newton distance of $S$ with respect to $(p_1,\dots, p_d)$. We claim that, up to $\log(\lambda)$ terms, if
	\begin{equation*}
		|\Lambda_d(f_1,\dots,f_d)|\lesssim \lambda^{-r}\prod_{j=1}^d\|f_j\|_{p_j}
	\end{equation*}
	uniformly in $f_j\in L^{p_j}$, then $r\le 1/\delta,$ and therefore $-1/\delta$ is the sharp exponent if such an estimate holds (in particular, the sharp exponent in Theorem \ref{thm: main theorem} whenever $\delta\ge 2^{d-1}$). Let $\bsb{n}$ be such that $\alpha\cdot\bsb{n}=1$, and for $1\le j\le d$ let 
	\begin{equation*}
		f_j(\hat{x}_j)=\prod_{k\neq j}^d\one_{[0,\lambda^{-n_k}]}(x_k).
	\end{equation*}
	Then for large $\lambda,$ up to logarithmic factors, 
	\begin{multline*}
		\prod_{j=1}^d\|f_j\|_{p_j}=\prod_{j=1}^d \lambda^{- (\one\cdot\bsb{n})/p_j+n_j/p_j}\\
		=\lambda^{-(P-1/p_1,\dots, P-1/p_d)\cdot\bsb{n}}=\lambda^{\alpha\cdot\bsb{n}/\delta -\one\cdot\bsb{n}} = \lambda^{1/\delta-\one\cdot\bsb{n}}.
	\end{multline*}
	One can also compute
	\begin{equation*}
		|\Lambda_d(f_1,\dots,f_d)|\approx \lambda^{- \one\cdot\bsb{n}}.
	\end{equation*}
	Therefore, the quotient, $\lambda^{-1/\delta}$, serves as a bound for the decay rate.

	We just proved Theorem \ref{thm: main theorem} is sharp in the exponent of $\lambda$ whenever $\delta\ge 2^{d-1}.$ For example, if $p_j=2^{j}$ for $1\le j\le d-1$ and $p_d=2^{d-1}$ in Theorem \ref{thm: main theorem}, hypothesis $(i)$, or if $p_j=p(d)$ for all $j$ under hypothesis $(ii)$, the decay in $\lambda$ is sharpest possible (over all nondegenerate phases $S$). However, the exponent of $\log$ is not sharp in general: if $\delta> 2^{d-1},$ Corollary \ref{cor: linop} asserts the exponent of $\log$ is even better than stated: it is $k-1$ where $k$ is the largest codimension over any face in $\cn(S)$ containing the vector $\delta(1-P+p_1^{-1},\dots, 1-P+p_d^{-1})$ (which is still not sharpest possible in general, e.g., the estimate for $S(x,y,z)=xyz$ has no $\log$ terms for the smallest exponents in Theorem \ref{thm: main theorem}).
	
	As the following proposition shows, one could interpolate any sharp estimate resulting from Theorem \ref{thm: main theorem} the Loomis-Whitney inequality with all exponents equal $d-1$ to obtain sharp decay estimates for other choices of exponents. However, this statement generally fails to hold for other versions of Loomis-Whitney on compact domains.
		
	\begin{proposition}\label{prop: interpolation sharpness}
		Assume that, up to logarithmic factors,
		\begin{equation*}
			|\Lambda_d(f_1,\dots,f_d)|\lesssim \lambda^{-1/\delta}\prod_{j=1}^d \|f_j\|_{p_j},
		\end{equation*}
		where $\delta$ is the Newton distance with respect to $(p_1,\dots, p_d).$ Assume there are $0<\theta_1, \theta_2<1$ summing to 1 such that
		\begin{equation*}
			q_j^{-1}=\theta_1p_j^{-1}+\theta_2(d-1)^{-1}.
		\end{equation*}
		Up to logarithmic factors,
		\begin{equation*}
			|\Lambda_d(f_1,\dots,f_d)|\lesssim \lambda^{-1/\delta'} \prod_{j=1}^d \|f_j\|_{q_j},
		\end{equation*}
		where $\delta'$ is the Newton distance with respect to $(q_1,\dots, q_d).$
	\end{proposition}
	
	The sharpness of the conclusion follows from previous examples.
	
	\begin{proof}
		By multilinear interpolation,
		\begin{equation*}
			|\Lambda_d(f_1,\dots,f_d)|\lesssim \lambda^{-\theta_1/\delta} \prod_{j=1}^d \|f_j\|_{q_j}.
		\end{equation*}
		
		To prove the claim, we simply need to show that $\delta'=\delta/\theta_1$ is the Newton distance with respect to $(q_1,\dots,q_d)$. 
		
		Letting $Q=\sum_{j=1}^d q_j^{-1},$ it is easy to see that $Q=\theta_1 P + \theta_2d(d-1)^{-1}$, and therefore
		\begin{multline*}
			1-Q+q_j^{-1}=1-\left(\theta_1 P + \theta_2d(d-1)^{-1}\right)+ \left(\theta_1p_j^{-1}+\theta_2(d-1)^{-1}\right)\\
			=\theta_1(1-P+p_j^{-1})+\theta_2(1-d(d-1)^{-1}+(d-1)^{-1})=\theta_1(1-P+p_j^{-1}).
		\end{multline*}
		Multiplying both sides by $\delta/\theta_1$ shows
		\begin{multline*}
			(\delta/\theta_1)\left(1-Q+q_1^{-1},\dots,1-Q+q_d^{-1}\right)\\
			= \delta\left(1-P+p_1^{-1},\dots,1-P+p_d^{-1}\right)\in\partial\cn(S),
		\end{multline*}
		establishing the claim. 
	\end{proof}
	We note that the choice of exponents uniformly equal to $d-1$ was essential in removing $\theta_2$ from the computation of $1-Q+q_j^{-1}$, since there is no reason for $\delta/\theta_1$ to be the Newton distance with respect to $(q_1,\dots,q_d)$ in general. For details about the interpolation theorem used above, see \cite{BeLo76}*{Chapter 4}.

	\subsection{Growth Estimates}\label{GE}
	
	\begin{proposition}
		Let $S:\R^d\to\R$ be real analytic in a neighborhood of the origin. For all $\beta\in\N^d$, there is a constant $C$ depending on $\beta$ and derivatives of $S$ such that for all $x$ close enough to the origin,
		\begin{equation*}
			|x^\beta \partial^\beta S(x)|\le C \max_{\alpha\in \cn(S)} x^\alpha.
		\end{equation*}
	\end{proposition}

	\begin{proof}
		Since $\cn(S)$ has finitely many vertices, let $N$ be such that $|\alpha|\geq N$ implies $\alpha$ is not a vertex of $\cn(S).$ Since $S$ is real analytic,
		\begin{equation*}
			S(x)=\sum_{|\alpha|\le N} c_\alpha x^\alpha + \sum_{|\alpha|=N}h_\alpha(x)x^\alpha
		\end{equation*}
		where $|h_\alpha(x)|\to 0$ as $x\to \bz.$ Therefore, for all $\beta\in \N^d,$ we also have a Taylor series expansion
		\begin{equation*}
			x^\beta\partial^\beta S(x) = \sum_{|\alpha|\le N} c_{\alpha,\beta}x^\alpha + \sum_{|\alpha|=N} h_{\alpha,\beta}(x)x^\alpha,
		\end{equation*}
		where $|h_{\alpha,\beta}(x)|\to 0$ as $x\to\bz.$ %For all $\alpha\in \cn(S), |\alpha|=N$ implies that $\alpha$ is not contained in a compact face of $\cn(S).$ Therefore for all $x\in[-1,1]^d,$ there is some $|\alpha'|\le N$ such that $|x^{\alpha'}|\ge |x^\alpha|.$ 
		Using this fact, for each $\beta$, for $x$ small enough, the error terms $h_{\alpha,\beta}$ are small enough such that
		\begin{equation*}
			|x^\beta\partial^\beta S(x)| \le 2\sum_{|\alpha|\le N} |c_{\alpha,\beta}x^\alpha|\le C_\beta\max_{\text{Vertices }\alpha\in\cn(S)}|x^\alpha|\le  C_\beta \max_{\alpha\in\cn(S)} |x^\alpha|.
		\end{equation*}
		The second inequality follows from the fact that among $\alpha\in\cn(S)$, the vertices of $\cn(S)$ maximize $|x^\alpha|$ for $|x_j|<1$.
	\end{proof}
	Taking a finite collection of derivatives leads to the following corollary:
	\begin{corollary}\label{cor:up bound}
		Let $\beta^1,\dots,\beta^k\in\N^d.$ There is a constant $C$ such that for all $x$ close enough to the origin, for all $1\le i\le k,$
		\begin{equation*}
			|\partial^{\beta^i} S(x)| \le  C\max_{\alpha\in\cn(S)} |x^{\alpha-\beta^i}|.
		\end{equation*}
	\end{corollary}	
	For the lower bound on derivatives of $S$, we refer to the following lemma.
	\begin{lemma}\label{lemma: lb}
		Assume $S:\R^d\to\R$ is nondegenerate. Then for all $\varep\in(0,1)^d$ close enough to the origin, there is a uniform constant $C$ independent of $\varep$ such that for all $x\in\prod_{j=1}^d[\varep_j,4\varep_j]$ and all $\alpha\in\cn(S),$
		\begin{equation*}
			|D_dS(x)|\geq C\varep^{\alpha-\bone}.
		\end{equation*}
	\end{lemma}
	
	The proof is nearly identical to the proof of \cite{Gi18}*{Lemma 2.1}, so we won't recreate it here. Moreover, \lq\lq close enough\rq\rq can be quantified by analyzing the argument in \cite{Gi18}. For a more heuristic argument, see \cite{GiGrXi18}*{Lemma 3.1}.

	\section{Local Estimates for \texorpdfstring{$d=3$}{d=3}}\label{sec: d=3}
	
	In this section, we prove local estimates for the $d=3$ case using the following local estimates for the $d=2$ case.
	
	For $f\in C_0^\infty (\R)$ define the oscillatory integral operator 
	\begin{equation*}
		T(f)(x)=\int_{-\infty}^\infty e^{i\lambda S(x,y)}\psi(x,y)f(y)dy,
	\end{equation*}
	where $\psi$ is a smooth compactly supported function on $\R^2.$ The key operator van der Corput lemma we use for the first step of our induction argument can be found in Greenblatt\cite{gr04}, and has been previously studied by H\"ormander\cite{ho71} and Phong-Stein\cite{ps94}\cite{ps94op} under various other hypotheses.
	\begin{lemma}[Operator van der Corput]\label{lemma: Phong-Stein}
		Consider the operator $T$ defined directly above. Assume that $S(x,y)$ is smooth and
		\begin{itemize}
			\item[(a)] $\psi$ is supported in a rectangle, having width $b$ in the $y-$direction;
			\item[(b)] there is a constant $C_1>0$ such that $|\partial_y^n\psi(x,y)|\le C_1 b^{-n}$ for $n=0,1,2$;
			\item[(c)] there are constants $C_2, \mu>0$ such that $S(x,y)$ satisfies 
			\begin{equation*}
				\Bigg{|}\frac{\partial^{2}S}{\partial x\partial y}\Bigg{|}\ge\mu,\hspace{1cm} \Bigg{|}\frac{\partial^{3}S}{\partial x\partial y^2}\Bigg{|}\le C_2\mu b^{-1},\text{ and } \Bigg{|}\frac{\partial^{4}S}{\partial x\partial y^3}\Bigg{|}\le C_2\mu b^{-2}
			\end{equation*} 
			in the support of $\psi$.
		\end{itemize}
		Then there is a constant $C$ depending only on $C_1, C_2$ such that for $f,g\in C_0^\infty(\R)$,
		\begin{equation*}
			\left|\langle Tf,g\rangle\right|=\Bigg{|}\int_{\R^2} e^{i\lambda S(x,y)}\psi(x,y) f(y)g(x)dxdy\Bigg{|} \leq C (\lambda\mu)^{-1/2}\|f\|_2\|g\|_2,
		\end{equation*}
	\end{lemma}
	
	In this section, we will prove a version of the above operator van der Corput lemma for $\Lambda_3$. In it, and in the higher dimensional versions found in the next section, we will need a generalization of the hypotheses of Lemma \ref{lemma: Phong-Stein}. For this purpose, we say that {\it $\Lambda_d$ satisfies the local estimate hypotheses in $x_j$ with parameter $\mu>0$} if 
	\begin{enumerate}
		\item $\psi(x)$ is supported on a rectangle of width $b_j$ in the $x_j-$direction.
		\item There is a constant $C_1>0$ such that $|\partial_{x_j}^n\psi(x)|\leq C_1b_j^{-n}$ for $n=0,1,2$.
		\item There exists constant $C_2>0$ such that
		\begin{equation*}
			|D_dS|\ge\mu,\hspace{1cm} |\partial_{x_j}D_dS|\le C_2\mu b_j^{-1},\text{ and } |\partial_{x_j}^2D_dS|\le C_2\mu b_j^{-2}
		\end{equation*} 
		in the support of $\psi$.
	\end{enumerate}
	
	Assuming the local estimate hypothesis leads to the following local theorem for $d=3:$
	
	\begin{theorem}\label{thm: mixed norm}
		Suppose that $\Lambda_3$ satisfies the local estimate hypotheses in $x_1$ with parameter $\mu>0$. 
		
		Then, there exists $C>0$ depending only on $C_1,C_2$, and $d$ such that
		\begin{equation*}
			|\Lambda_3(f_1,f_2,f_3)|\le C (\lambda\mu)^{-1/4}\|f_1\|_{L^4_{x_1}L^{8/3}_{x_3}}\|f_2\|_{L^4_{x_2}L^{8/3}_{x_3}}\|f_3\|_{2}.
		\end{equation*}
	\end{theorem}
	
	\begin{proof}
		Applying the $TT^*$-method to the function $f_3$, we have
		\begin{multline}\label{eq:tt*}
			|\Lambda_3(f_1,f_2,f_3)|^2\leq\|f_3\|_2^2\iiiint e^{i\lambda(S(x_1,x_2,x_3)-S(x_1,x_2,y))}
			f_1(x_2,x_3)\\
			\times\overline{f_1}(x_2,y)f_2(x_1,x_3)\overline{f_2}(x_1,y)\psi(x_1,x_2,x_3)\overline{\psi}(x_1,x_2,y)dx_1dx_2dx_3dy.
		\end{multline}
		
		Let 
		\begin{align*}
			S_{x_3,y}(x_1,x_2)&=S(x_1,x_2,x_3)-S(x_1,x_2,y),\\
			F_{1,x_3,y}(x_2)&=f_1(x_2,x_3)\overline{f_1}(x_2,y),\\
			F_{2,x_3,y}(x_1)&=f_2(x_1,x_3)\overline{f_2}(x_1,y),\text{ and let}\\ \psi_{x_3,y}(x_1,x_2)&=\psi(x_1,x_2,x_3)\overline{\psi}(x_1,x_2,y).
		\end{align*}. 
		Thus, the right hand side of \eqref{eq:tt*} becomes
		\begin{equation}\label{eq: post sub}
			\iint\left[\iint e^{i\lambda S_{x_3,y}(x_1,x_2)}
			F_{1,x_3,y}(x_2) F_{2,x_3,y}(x_1)\psi_{x_3,y}(x_1,x_2)dx_1dx_2\right]dx_3dy.
		\end{equation}
		
		We want to apply the operator van der Corput lemma to \eqref{eq: post sub} in the $(x_1,x_2)$ variables for each $(x_3,y)$. In order to do so, we must check that the local estimate hypotheses for $S$ imply local estimate hypotheses for $S_{x_3,y}$.
		
		First, it is clear that if $\psi(x_1,x_2,x_3)$ has support contained in a rectangle of width $b_1$ in the $x_1$-direction, then so does $\psi_{x_3,y}$. Second, the smoothness of $\psi$ guarantees uniformity of the constant $C_1$. Third, we see that
		\begin{equation*}
			\partial_{x_1}\partial_{x_2}S_{x_3,y}(x_1,x_2)=\int_y^{x_3} D_3S(x_1,x_2,t)dt,
		\end{equation*}
		so by the Mean Value Theorem, $|D_2S_{x_3,y}(x_1,x_2)|\geq \mu|x_3-y|$ uniformly in $x_1,x_2$. Bounds for the $x_1$-derivatives of $S_{x_3,y}$ follow similarly. Thus, $S_{x_3,y}$ satisfies the local estimate hypotheses with parameter $\mu|x_3-y|$.
		
		Now applying Lemma \ref{lemma: Phong-Stein}, we control \eqref{eq: post sub} by
		\begin{equation*}
			\iint|\lambda\mu(x_3-y)|^{-1/2}\|F_{1,x_3,y}\|_{L^2_{x_2}}\|F_{2,x_3,y}\|_{L^2_{x_1}}dx_3dy.
		\end{equation*}
		By the Cauchy-Schwarz inequality, this is controlled by
		\begin{multline*}
			(\lambda\mu)^{-1/2}\left(\iint|x_3-y|^{-1/2}\|F_{1,x_3,y}\|_{L^2_{x_2}}^2dx_3dy\right)^{1/2}\\ 
			\times\left(\iint|x_3-y|^{-1/2}\|F_{2,x_3,y}\|_{L^2_{x_1}}^2dx_3dy\right)^{1/2}.
		\end{multline*}
		Both integral factors above are essentially the same, so we focus our attention on the first. By Fubini's theorem and the Hardy-Littlewood-Sobolev inequality,
		\begin{align*}
			\begin{split}
				\iint|x_3-y|^{-1/2}\|F_{1,x_3,y}&\|_{L^2_{x_2}}^2dx_3dy\\
				&=\int\left(\iint |x_3-y|^{-1/2}f_1(x_2,x_3)^2\overline{f_1}(x_2,y)^2dx_3dy\right)dx_2
			\end{split}\\
			&\leq\int \|f_1^2(x_2,x_3)\|_{L^{4/3}_{x_3}}^2dx_2\\
			&=\int \|f_1(x_2,x_3)\|_{L^{8/3}_{x_3}}^4dx_2\\
			&=\|f_1\|_{L^4_{x_1}L^{8/3}_{x_3}}^4.
		\end{align*}
		
		Plugging the norm bounds back in to \eqref{eq:tt*} and taking square roots, the estimate of Theorem \ref{thm: mixed norm} is established.
	\end{proof}
	
	Theorem \ref{thm: mixed norm} has the following consequences. 
	
	\begin{corollary}\label{cor:off diag d=3}
		Suppose that $\Lambda_3$ satisfies the local estimate hypotheses in $x_1, x_2$, and $x_3$ with parameter $\mu$, and furthermore, that $\psi$ is supported on a rectangle of dimensions $b_1\times b_2\times b_3$.
		
		Then, 
		\begin{equation*}
			|\Lambda_3(f_1,f_2,f_3)|\le(b_3)^{1/4} (\lambda\mu)^{-1/4}\|f_1\|_{4}\|f_2\|_{4}\|f_3\|_{2}.
		\end{equation*}
	\end{corollary}
	
	The above will be useful in establishing off-diagonal estimates in higher dimensions. While it is proven through a simple application of H\"older's inequality (which typically sacrifices sharpness), the result is actually sharp in its decay in $\lambda$. (Take $b_j\equiv1$, for instance.) An example demonstrating sharpness will be provided in Section \ref{sec: higher d}.
	
	\begin{corollary}\label{cor:on diag d=3}
		Suppose that $\Lambda_3$ satisfies the local estimate hypotheses.
		
		Then, 
		\begin{equation*}
			|\Lambda_3(f_1,f_2,f_3)|\le (\lambda\mu)^{-1/4}\|f_1\|_{8/3}\|f_2\|_{8/3}\|f_3\|_{8/3}.
		\end{equation*}
	\end{corollary}
	
	Corollary \ref{cor:on diag d=3} is proven through a simple application of interpolation in mixed norm spaces, though we provide full details here.
	
	\begin{proof}
		
		By permutation of indices and Minkowski's inequality, 
		\begin{equation*}
			|\Lambda_3(f_1,f_2,f_3)|\lesssim	(\lambda\mu)^{-1/4}
			\begin{cases}
				\|f_1\|_2\|f_2\|_{L^{8/3}_{x_1}L^4_{x_3}}\|f_3\|_{L^{8/3}_{x_1}L^4_{x_2}}\\
				\|f_2\|_2\|f_3\|_{L^4_{x_1}L^{8/3}_{x_2}}\|f_1\|_{L^{8/3}_{x_2}L^4_{x_3}}\\
				\|f_3\|_2\|f_1\|_{L^4_{x_2}L^{8/3}_{x_3}}\|f_2\|_{L^4_{x_1}L^{8/3}_{x_3}}.
			\end{cases}
		\end{equation*}
		Duality for mixed norm spaces guarantees $L^{p'}L^{q'}$ is dual to $L^pL^q$, where $p',q'$ are the dual exponents to $p,q$, respectively. Therefore the bilinear operator $T$ defined by
		\begin{equation*}
			T(f_2,f_3)(x_1)=\iint e^{i\lambda S(x_1,x_2,x_3)}f_2(x_1,x_3)f_3(x_1,x_2)\psi(x_1,x_2,x_3)dx_2dx_3
		\end{equation*}
		is a bounded map from
		\begin{align*}
			L^4L^{\frac{8}{3}}\times L^4L^{\frac{8}{3}}&\to L^{2}L^{2}\\
			L^{\frac{8}{3}}L^4\times L^2L^{2}&\to L^{\frac{8}{5}}L^{\frac{4}{3}},\text{ and }\\
			L^2L^2\times L^{\frac{8}{3}}L^4&\to L^{\frac{4}{3}}L^{\frac{8}{5}},
		\end{align*}
		with norm $\lesssim(\lambda\mu)^{-1/4}.$ Applying interpolation for bilinear operators on mixed norm spaces, we conclude that $T:L^{8/3}\times L^{8/3}\to L^{8/5},$ and therefore
		\begin{multline*}
			|\Lambda_3(f_1,f_2,f_3)|\lesssim (\lambda\mu)^{-1/4}\|f_1\|_{8/3}\|f_2\|_{8/3}\|f_3\|_{8/3} \\
			=(\lambda\mu)^{-1/4}\|f_1\|_{8/3}\|f_2\|_{8/3}\|f_3\|_{8/3}.
		\end{multline*}
		The details for this interpolation theorem can be found in \cite{BeLo76}*{Theorem 5.1.2}.
	\end{proof}
	
	\section{Local Estimates in Higher Dimensions}\label{sec: higher d}
	
	\subsection{The Off-Diagonal Case}
	
	\begin{theorem}\label{thm: off diagonal}
		Suppose $\Lambda_d$ satisfies the local estimate hypothesis in $x_1$ with parameter $\mu$ and that furthermore, $\psi$ is supported on a box of dimensions $b_1\times\dots\times b_d$. Then
		\begin{equation*}
			|\Lambda_d(f_1,\dots,f_n)|\leq C(\lambda\mu)^{-2^{1-d}}\prod_{j=1}^db_j^{1/p_j-2^{1-d}}\|f_j\|_{L^{p_j}}, 
		\end{equation*}
		where $(p_1,\dots,p_d)=(2^{d-1},2^{d-1},2^{d-2},\dots,4,2)$ and $C$ depends solely on $d$ and the $C_1,C_2$ given in the local estimate hypotheses.
	\end{theorem}
	
	The following example shows Theorem \ref{thm: off diagonal} is sharp. Let $S(x)=\mu x_1\cdots x_d$ and $f_j=\one_{[0,|\lambda\mu|^{-1/2}]}(x_{d-1})\one_{[0,|\lambda\mu|^{-1/2}]}(x_{d})$ when $1\leq j\leq d-2$, while $f_{d-1}=\one_{[0,|\lambda\mu|^{-1/2}]}(x_{d})$ and $f_d=\one_{[0,|\lambda\mu|^{-1/2}]}(x_{d-1})$. Then, an elementary scaling argument shows $|\Lambda_d|\approx|\lambda\mu|^{-1}$ and $\prod_j\|f_j\|_{p_j}\approx|\lambda\mu|^{-2^{1-d}-1}$.
	\begin{proof}
		
		The proof is by induction on $d$, where the $d=3$ case is established by Theorem \ref{thm: mixed norm}. 
		
		Suppose that Theorem \ref{thm: off diagonal} holds in dimension $d$. We wish to show it holds in dimension $d+1$. Here, we write $x'=(x_1,\dots,x_d)$, so $x=(x',x_{d+1})$.
		
		Applying the $TT^*$ method to $\Lambda_{d+1}$ on the function $f_{d+1}$, we obtain
		\begin{multline}\label{eq:tt* off diag}
			|\Lambda_{d+1}(f_1,\dots,f_{d+1})|^2\leq\|f_{d+1}\|_2^2\iiint e^{i\lambda(S(x',x_{d+1})-S(x',y))}\\
			\times\prod_{j=1}^d f_j(\pi_j(x',x_{d+1}))\overline{f_j}(\pi_j(x',y))\psi(x',x_{d+1})\overline{\psi}(x',y)dx'dx_{d+1}dy.
		\end{multline}

		Define $\hat{x}'_j:=(x_1,\dots,x_{j-1},x_{j+1},\dots,x_d)$. Let
		\begin{align*}
			S_{x_{d+1},y}(x')&=S(x',x_{d+1})-S(x',y),\\ 
			F_{j,x_{d+1},y}(\hat{x}_j')&=f_j(\pi_j(x',x_{d+1}))\overline{f_j}(\pi_j(x',y))\text{ for }1\leq j\leq d,\text{ and }\\
			\psi_{x_{d+1},y}(x)&=\psi(x',x_{d+1})\overline{\psi}(x',y),
		\end{align*}
		to rewrite the right hand side of \eqref{eq:tt* off diag} as 
		\begin{equation}\label{eq: post sub II}
			\iiint e^{i\lambda S_{x_{d+1},y}(x')}\prod_{j=1}^d F_{j,x_{d+1},y}(\hat{x}_j')\psi_{x_{d+1},y}(x')dx'dx_{d+1}dy.
		\end{equation}
		
		At this point, we would like to apply the induction hypothesis. As in the case of applying Lemma \ref{lemma: Phong-Stein} in the proof of the $d=3$ estimates, it is trivial that $\psi_{x_{d+1},y}$ is supported in a rectangle of width $b_1$ in the $x_1$-direction, and the smoothness of $\psi$ guarantees uniformity of the constant $C_2$.
		
		As before, we see that
		\begin{equation*}
			D_dS_{x_{d+1},y}(x)=\int_y^{x_{d+1}} D_{d+1}S(x,t)dt,
		\end{equation*}
		so by the Mean Value Theorem, $|D_dS_{x_{d+1},y}(x')|\geq \mu|x_{d+1}-y|$ uniformly in $x'$. Bounds for the $x_1$-derivatives of $S_{x_{d+1},y}$ follow similarly. Thus, the local estimate hypotheses hold with parameter $\mu|x_{d+1}-y|$.
		
		We now apply the induction hypothesis in $x$ to control \eqref{eq: post sub II} by
		\begin{equation*}
			(\lambda\mu)^{-2^{1-d}}\iint |x_{d+1}-y|^{-2^{1-d}}\prod_{j=1}^db_j^{1/p_j-2^{1-d}}\|F_{j,x_{d+1},y}\|_{L^{p_j}_{\hat{x}_j}}dx_{d+1}dy
		\end{equation*}
		where $(p_1,\dots,p_d)=(2^{d-1},2^{d-1},2^{d-2},\dots,4,2)$. Note that $\sum_{j=1}^dp_j^{-1}=1$. By H\"older's inequality, this is bounded by
		\begin{multline*}
			(\lambda\mu)^{-2^{1-d}}\prod_{j=1}^db_j^{1/p_j-2^{1-d}}\left(\iint|x_{d+1}-y|^{-2^{2-d}}\|F_{d,x_{d+1},y}\|^2_{L^{2}_{\hat{x}_d}}dx_{d+1}dy\right)^{1/2}\\
			\times\prod_{j=1}^{d-1}\left(\iint\|F_{j,x_{d+1},y}\|_{L^{p_j}_{\hat{x}_j}}^{p_j}dx_{d+1}dy\right)^{1/p_j}.
		\end{multline*}
		
		Using the same steps as in the end of the proof of Theorem \ref{thm: mixed norm}, we bound the first factor by $\|f_d\|^2_{L^4_{x_1,\dots,x_{d-1}}L^{2/(1-2^{1-d})}_{x_{d+1}}}$, which is controlled by $b_{d+1}^{(1-2^{2-d})/2}\|f_d\|_{4}^2,$ since $1-2^{1-d}=\frac{1}{2}+\frac{1-2^{2-d}}{2},$ and then both terms are square rooted.
		
		The remaining factors are bounded as follows:
		\begin{multline*}
			\left(\iint\|F_{j,x_{d+1},y}\|_{L^{p_j}}^{p_j}dx_{d+1}dy\right)^{1/p_j}=\\
			\left(\iiint f_j^{p_j}(\hat{x}'_j,x_{d+1})\overline{f_j}^{p_j}(\hat{x'}_j,y)d\hat{x}'_jdx_{d+1}dy\right)^{1/p_j}\\
			=\|f_j\|^2_{L^{2p_j}_{\hat{x}'_j}L^{p_j}_{x_{d+1}}},
		\end{multline*}
		which is controlled by $b_{d+1}^{1/p_{j}}\|f_j\|^2_{L^{2p_j}}$.
		
		We note that cumulatively, we obtain the exponent of $b_{d+1}$ above as $(1/2-2^{1-d})+1/p_2+\cdots +1/p_d=1-2^{1-d}=2(1/2-2^{1-(d+1)}).$ The exponent of $b_j$ is simply $1/p_j.$ Theorem \ref{thm: off diagonal} follows from taking square roots. (All of the exponents from the dimension $d$ case double, and the exponent 2 comes from the use of $TT^*$.)
		
	\end{proof}

	\subsection{The On-Diagonal Case}
	
	\begin{theorem}\label{thm: on diagonal}
		Suppose $\Lambda_d$ satisfies the local estimate hypotheses in $x_1,\dots,x_d$ with parameter $\mu$. 
		
		Then,
		\begin{equation*}
			|\Lambda_d(f_1,\dots,f_d)|\leq C(\lambda\mu)^{-2^{1-d}}\prod_{j=1}^d\|f_j\|_{L^{p(d)}},
		\end{equation*}
		where $p(d):=(d-1)\frac{2^{d-1}}{2^{d-1}-1}$ and and $C$ depends solely on $d$ and the $C_1,C_2$ given in the local estimate hypotheses.
	\end{theorem}	
	
	This estimate is also sharp, as may be seen by taking $S(x)=\mu x_1\cdots x_d$ and $f_j=\one_{B(0,|\lambda\mu|^{-1/d})}$, where $B(0,r)$ denotes the ball of radius $r$ centered at the origin in $\R^{d-1}$.

	\begin{proof}[Proof of Theorem \ref{thm: on diagonal}]
		We pick up at equation \eqref{eq:tt* off diag} in the middle of the proof of Theorem \ref{thm: off diagonal}, where we replace the previous induction hypothesis with Theorem \ref{thm: on diagonal} in the dimension $d$ case. (We will repeat these skipped steps for every permutation of indices, which is allowed because the local estimate hypotheses hold in each coordinate direction.)
		
		Applying this induction hypothesis to \eqref{eq:tt* off diag} in $x$ with $p=p(d)$, we obtain a bound of 
		\begin{equation*}
			(\lambda\mu)^{-2^{1-d}}\iint |x_{d+1}-y|^{2^{1-d}}\prod_{j=1}^d\|F_{j,x_{d+1},y}\|_{L^{p(d)}_{\hat{x}'_j}}dx_{d+1}dy.
		\end{equation*}

		From here on, write $f_j=f_j(w,x_{d+1})$, where $w\in\R^{d-1}$. Applying H\"older's inequality, we get a bound of $\lambda^{-2^{1-d}}$ times
		\begin{equation*}
			\left(\iint|x_{d+1}-y|^{-p2^{1-d}}\|F_{1,x_{d+1},y}\|_{L^p_w}^pdx_{d+1}dy\right)^{1/p}\prod_{j=2}^d\|F_{j,x_{d+1},y}\|_{L^q_{x_{d+1},y}L^p_w},
		\end{equation*}
		where $q=\frac{2^{d-1}(d-1)^2}{2^{d-1}(d-2)+1}$ is defined via the equation $\frac{1}{p}+(d-1)\frac{1}{q}=1$.
		
		For the first integral factor, we use the same Hardy-Littlewood-Sobolev estimate in the $x_{d+1}$ and $y$ variables to get a bound of
		\begin{equation*}
			\|f_1\|^2_{L_w^{2p}L_{x_{d+1}}^r},
		\end{equation*}
		where $r=\frac{2p}{1-p2^{1-d}/2}=\frac{2p}{1-p2^{-d}}=\frac{2^d(d-1)}{2^d-d-1}$.
		
		For the second factor, we observe $p<q$ and apply Minkowski's inequality to get
		\begin{align*}
			\|F_{j,x_{d+1},y}\|_{L^q_{x_{d+1},y}L^p_w}&\leq\|F_{j,x_{d+1},y}\|_{L^p_wL^q_{x_{d+1},y}}\\
			&=\|f_j(w,x_{d+1})\overline{f_j}(w,y)\|_{L^p_wL^q_{x_{d+1},y}}\\
			&\leq\Big{\|}\ \|f_j\|_{L^q_{x_{d+1}}}^2\Big{\|}_{L^p_w}\\
			&=\|f_j\|_{L^{2p}_wL^q_{x_{d+1}}}^2.
		\end{align*}
		Putting it all together, we get an estimate of
		\begin{equation}\label{eq: to interpolate}
			|\Lambda_{d+1}(f_1,\dots,f_{d+1})|\leq \lambda^{-2^{-d}}\|f_{d+1}\|_{2}\|f_1\|_{L_w^{2p}L_{x_{d+1}}^r}\prod_{j=2}^d\|f_j\|_{L^{2p}_wL^q_{x_{d+1}}}.
		\end{equation}
		
		One may reach the conclusion of Theorem \ref{thm: on diagonal} by repeating the above argument for all permutations of $1,\dots,d+1$ and interpolating, applying Minkowski's inequality when necessary. (Observe that all of the mixed norms above are eligible for use of Minkowski since the norm with smaller exponent is taken first.)
		
		We now provide the details which show how $p(d+1)$ comes out of the interpolations.
		
		The norm bounds in \eqref{eq: to interpolate} are determined by the choices of indices for the $L^2$ bound and for the $L^{2p}L^r$ bound, then the rest are fixed. Thus, there are $d(d+1)$ total bounds to interpolate. We interpolate with all of them to obtain the same exponent for each $f_j$.
		
		For any particular $f_j$, we get $d$ many exponents 2, an $L^{2p}L^r$ bound $d$ times (with each variable $x_1,\dots,x_{j-1},x_{j+1},\dots,x_{d+1}$ falling under the $L^r$ norm once), and an $L^{2p}L^q$ bound $d(d-1)$ times (with the same variables falling under the $L^q$ norm $d-1$ times). Here $p$ refers to the exponent $p(d)=\frac{2^{d-1}(d-1)}{2^{d-1}-1}$, which comes from induction.
		
		Note that $2p(d)=\frac{2^d(d-1)}{2^{d-1}-1}$. Thus, the calculation for $1/p(d+1)$ is
		\begin{align*}
			\frac{1}{p(d+1)}&=\frac{1}{d+1}\frac{1}{2}+\frac{1}{d+1}\left(\frac{d-1}{d}\frac{1}{2p(d)}+\frac{1}{d}\frac{1}{r}\right)+\frac{d-1}{d+1}\left(\frac{d-1}{d}\frac{1}{2p(d)}+\frac{1}{d}\frac{1}{q}\right)\\
			&=\frac{1}{2(d+1)}+\frac{d-1}{d+1}\frac{1}{2p(d)}+\frac{1}{d(d+1)}\frac{1}{r}+\frac{d-1}{d(d+1)}\frac{1}{q}\\
			\begin{split}
				&=\frac{1}{2(d+1)}+\frac{d-1}{d+1}\frac{1}{2p(d)}+\frac{1}{d(d+1)}\frac{2^d-d-1}{2^d(d-1)}\\
				&\hspace{6cm}+\frac{d-1}{d(d+1)}\frac{2^{d-1}(d-2)+1}{2^{d-1}(d-1)^2}
			\end{split}\\
			&=\frac{2^{d-1}d(d-1)+(2^{d-1}-1)d(d-1)+2^d-d-1+2^d(d-2)+2}{2^dd(d+1)(d-1)}\\
			%&=\frac{2^{d-1}d^2-2^{d-1}d+2^{d-1}d^2-2^{d-1}d-d^2+d+2^d-d-1+2^dd-2^{d+1}+2}{2^dd(d+1)(d-1)}\\
			&=\frac{2^dd^2-d^2-2^d+1}{2^dd(d+1)(d-1)}=\frac{2^d-1}{2^dd},
		\end{align*}
		which is indeed given by our formula for $p(d+1)$ in Theorem \ref{thm: on diagonal}.
	\end{proof}

	\section{Linear Optimization and Proof of Main Theorem}\label{sec: main proof}
	
	Let $\eta:\R\to [0,1]$ be supported in $[1,4].$ We can construct a dyadic partition of unity of $(0,1)^d$ by using functions supported in $[\varep_1, 4\varep_1]\times\cdots\times [\varep_d,4\varep_d]$ of the form $\eta_\varep(x)=\prod_{j=1}^d\eta(\varep_j^{-1}x_j)$, where $\varep_i=2^{-j_i}$ and $\vec{j}:=(j_1,\dots,j_d)$ ranges over $\N^d$. (From here on, we use $[-1,1]^d$ as the \lq\lq small enough\rq\rq neighborhood of the origin for simplicity.) This is a standard construction; see for example \cite{musc13}*{ch. 8}. Now let $\psi_\varep=\psi\eta_\varep.$ Let $\psi$ be supported close enough to the origin so that the hypotheses of Lemma \ref{lemma: lb} are satisfied, then whenever $x$ is in the support of $\psi_\varep,$
	\begin{equation*}
		|\partial_{k}^n \psi_\varep(x)|\lesssim \varep_k^{-n},
	\end{equation*}
	where the implicit constant only depends on derivatives of $\psi$ and $\eta.$ We will see that by the estimates in Section \ref{GE}, Theorems \ref{thm: off diagonal} and \ref{thm: on diagonal} hold with cutoff function $\psi_\varep$ with $\mu=\varep^{\alpha-\one}$ for all $\alpha\in\cn(S).$ Splitting up $\psi$ into $\sum \psi_\varep$,  by the triangle inequality, we consider the estimates of each $\Lambda_d$ with cutoff $\psi_\varep$ only in the orthant $[0,1]^d$. The strategy will be to sum these estimates to obtain the conclusion of Theorem \ref{thm: main theorem}. We now break up this summing method into a few separate results. We begin with the following extensions of the above theorems on $\varep-$boxes with cutoff $\psi_\varep.$ In particular, we can just think of each $f_j$ as being supported on the box $B(\hat{\varep}_j):=[\varep_1,4\varep_1]\times\cdots\times[\varep_{j-1},4\varep_{j-1}]\times[\varep_{j+1},4\varep_{j+1}]\times\cdots\times[\varep_d,4\varep_d]$.
	
	\begin{corollary}[Extending Off-Diagonal Exponents]\label{cor: extending off}
		Let $\psi$ and $S$ be as in Theorem \ref{thm: off diagonal}, and let $f_j$ be supported on the box $B(\hat{\varep}_j)$ for $1\le j\le d$. Let $p_1\ge 2^{d-1}, p_2\ge 2^{d-1}, p_3\ge 2^{d-2},\dots, p_{d}\ge 2.$ Then, there exists $C>0$ independent of $\varep$ such that for all $\alpha\in\cn(S)$,
		\begin{equation*}
			|\Lambda_d(f_1,\dots, f_d)| \lesssim (\lambda\varep^{\alpha})^{-2^{1-d}} \prod_{j=1}^{d} \varep_j^{1-P+p_j^{-1}}\|f_j\|_{p_j}.
		\end{equation*}
	\end{corollary}
	\begin{proof}
		Since $f_j$ is supported on $B(\hat{\varep}_j)\subset \R^{d-1},$ then for $p\ge q,$
		\begin{equation*}
			\|f_j\|_{q}\le (\varep_1\cdots\varep_{j-1}\varep_{j+1}\cdots \varep_d)^{q^{-1}-p^{-1}} \|f_j\|_p.
		\end{equation*}
		Let $(q_1,\dots,q_d)=(2^{d-1},2^{d-1},2^{d-2},\dots,2)$, and note $Q:=\sum_{j=1}^dq_j^{-1}=1$. By Lemma \ref{lemma: lb}, $|D_dS|\gtrsim\varep^{\alpha-\bone}$. Thus, by Theorem \ref{thm: off diagonal}, 
		\begin{align*}
			|\Lambda_d(f_1,\dots,f_d)|&\lesssim(\lambda\varep^{\alpha-\bone})^{-2^{1-d}}\prod_{j=1}^d\varep_j^{q_j^{-1}-2^{1-d}}\|f_j\|_{q_j}\\
			&=(\lambda\varep^\alpha)^{-2^{1-d}}\prod_{j=1}^d\varep_j^{q_j^{-1}}\|f_j\|_{q_j}\\
			&\lesssim(\lambda\varep^\alpha)^{-2^{1-d}}\prod_{j=1}^d(\varep_1\cdots\varep_{j-1}\varep_{j+1}\cdots \varep_d)^{q_j^{-1}-p_j^{-1}}\varep_j^{q_j^{-1}}\|f_j\|_{p_j}\\
			&=(\lambda\varep^{\alpha})^{-2^{1-d}} \prod_{j=1}^{d} \varep_j^{1-P+p_j^{-1}}\|f_j\|_{p_j}.
		\end{align*}
		
		Note that in our application of Theorem \ref{thm: off diagonal}, we obtain a constant dependent solely on $d$, and the $C_1,C_2$ given in the local estimate hypotheses. Corollary \ref{cor:up bound} applied to $D_dS$ in place of $S$ guarantees that $C_2$ is independent of $\varep$. The uniformity of $C_1$ for each choice of $\varep$ follows from the smoothness of $\psi$ and the scaling properties of the $\eta_\varep$ used in defining $\psi_\varep$.
	\end{proof}
	
	Let $q_j$ be as in the proof of Corollary \ref{cor: extending off}. By ignoring oscillation and simply applying H\"older's inequality 
	\begin{align*}
		|\Lambda_d(f_1,\dots,f_d)|&\le \left|\int \prod_{j=1}^d f_j(\hat{x}_j)\bone_{[\varep_j,4\varep_j]}(x_j)dx\right|\\
		&= \prod_{j=1}^{d} \varep_j^{q_j^{-1}}\|f_j\|_{q_j}\\
		&\lesssim \prod_{j=1}^{d} \varep_j^{q_j^{-1}}(\varep_1\cdots\varep_{j-1}\varep_{j+1}\cdots \varep_d)^{q_j^{-1}-p_j^{-1}}\|f_j\|_{p_j}\\
		&= \prod_{j=1}^d \varep_j^{1-P+p_j^{-1}} \|f_j\|_{p_j}.
	\end{align*}
	
	\begin{corollary}[Extending On-Diagonal Exponents]\label{cor: extending on}
		With the same setup as Theorem \ref{thm: on diagonal}, let $f_j$ be supported on the box $B(\hat{\varep}_j)$ for $1\le j\le d$.  Let $p_j\ge p(d)= (d-1)\frac{2^{d-1}}{2^{d-1}-1}.$ Then, there exists $C>0$ independent of $\varep$ such that for all $\alpha\in\cn(S)$,
		\begin{equation*}
			|\Lambda_d(f_1,\dots, f_d)| \lesssim (\lambda\varep^{\alpha})^{-2^{1-d}} \prod_{j=1}^{d} \varep_j^{1-P+p_j^{-1}}\|f_j\|_{p_j}.
		\end{equation*}
	\end{corollary}
	\begin{proof}
		The proof is identical to that of Corollary \ref{cor: extending off}, except that it relies on the estimates of Theorem \ref{thm: on diagonal} in place of Theorem \ref{thm: off diagonal}.
	\end{proof}
	
	We have another local estimate coming from Loomis-Whitney inequality for any $p_j\ge d-1$:
	\begin{multline*}
		|\Lambda_d(f_1,\dots,f_d)| \le \prod_{j=1}^{d}\|f_j\|_{d-1}\le \prod_{j=1}^{d}(\varep_1\cdots\varep_{j-1}\varep_{j+1}\cdots \varep_d)^{(d-1)^{-1}-p_j^{-1}}\|f_j\|_{p_j}\\
		= \prod_{j=1}^d \varep_j^{1-P+p_j^{-1}} \|f_j\|_{p_j}.
	\end{multline*}
	Briefly, we just established the following whenever $f_j$ are supported in $B(\hat{\varep}_j)$, for $p_j\ge p(d)$ or $p_1\geq 2^{d-1},p_2\ge 2^{d-1},p_3\ge 2^{d-2},\dots,p_d\geq2$:
	\begin{multline}\label{eq: higher size estimate}
		\left|\int_{\R^d}e^{i\lambda S(x)}\prod_{j=1}^df_j(\hat{x}_j)\psi_\varep(x)dx\right|\\
		\lesssim \min_{\alpha\in\cn(S)}\left\lbrace \prod_{j=1}^d \varep_j^{1-P+p_j^{-1}}, \lambda^{-2^{1-d}}\prod_{j=1}^d \varep_j^{-2^{1-d}\alpha_j+1-P+p_j^{-1}}\right\rbrace\prod_{j=1}^d \|f_j\|_{p_j}.
	\end{multline}
	
	Given $(p_1,\dots,p_d),$ the most important element of $\cn(S)$ we need to consider is $\alpha=\delta(1-P+p^{-1},\dots, 1-P+p_d^{-1}).$ If $k$ is the largest codimension over any face containing $\alpha,$ then $\alpha$ can be written as a convex combination $\sum_{i=1}^{d-k+1}\theta_i\alpha^i$ where $\alpha^i\in\partial\cn(S)$ are linearly independent, and the interplay between $\alpha, \delta, 2^{d-1},$ and the $\alpha^i$ all play a role in the final estimate of $\Lambda_d.$ We begin with the exact sum that needs bounding, move on to an optimization lemma, then have Corollary \ref{cor: linop} describe the interplay stated in the previous sentence.
	
	Summing over all $\varep-$boxes, we will need to estimate the sum over equation \eqref{eq: higher size estimate}:
	\begin{equation*}
		\sum_{j_1=1}^\infty\cdots\sum_{j_d=1}^\infty \min_{\alpha\in\cn(S)}\left\lbrace \prod_{i=1}^d \varep_{i}^{1-P+p_i^{-1}}, \lambda^{-2^{1-d}}\prod_{i=1}^d \varep_{i}^{-2^{1-d}\alpha_i+1-P+p_i^{-1}}\right\rbrace.
	\end{equation*}
	Starting with $j_k=1$ provides an upper bound since our $\varep$ needs to be small enough. If any $j_k\ge \log(\lambda)/(\delta(1-P+p_k^{-1})):=c_k\log(\lambda),$ then this part of the sum is bounded above by
	\begin{equation*}
		\sum_{j_1=1}^\infty\cdots\sum_{j_k=c_k\log(\lambda)}^\infty\cdots\sum_{j_d=1}^\infty \prod_{i=1}^d \varep_{i}^{1-P+p_i^{-1}}\le \sum_{j_k=c_k\log(\lambda)}^\infty 2^{-j_k(1-P+p_k^{-1})}
	\end{equation*}
	since each $1-P+p_k^{-1}>0.$ This is a geometric series bounded by $\lambda^{-1/\delta}.$ The remaining sum where all $j_i\le c_i\log(\lambda)$ is handled with the results below.
	
	\begin{lemma}\label{lem: lin op general}
		Let $\alpha^1,\dots,\alpha^k\in\R^d$ be such that $\alpha^1,\dots,\alpha^k,\bee_{k+1},\dots,\bee_d$ are linearly independent. Let $\alpha^{k+1}$ be such that 
		\begin{equation*}
			\sum_{i=1}^{k+1}\theta_i\alpha^i=\bz
		\end{equation*}
		for $\theta_i>0$ summing to 1. Let $M_1,\dots, M_{k+1}$ be strictly positive constants. Then
		\begin{equation*}
			\sum_{j_1\in\Z}\cdots\sum_{j_k\in\Z}\sum_{j_{k+1}=a_{k+1}}^{b_{k+1}}\cdots \sum_{j_d=a_d}^{b_d}\min_{1\le n\le k+1}\{M_n 2^{\alpha^n\cdot \vec{j}}\}\lesssim \prod_{i=k+1}^d(b_i-a_i)\prod_{i=1}^{k+1} M_i^{\theta_i},
		\end{equation*}
		where the implicit constant is independent of all $a_n, b_n,$ and $M_n.$
	\end{lemma}
	\begin{proof}
		The above quantity is summable since $\bz$ is in the convex hull of the $\alpha^n.$ The first step in the proof of this lemma is bounding the sum above by a uniform constant times
		\begin{equation*}
			\int_{-\infty}^\infty\cdots\int_{-\infty}^\infty\int_{a_{k+1}}^{b_{k+1}}\cdots\int_{a_d}^{b_d}\min_{1\le n\le k+1}\{M_n e^{\alpha^n\cdot x}\}dx_d\cdots dx_1.
		\end{equation*}
		(We actually obtain factors of $\log2$ in the exponents, though we ignore these factors as they are simply absorbed into the underlying, arbitrary constant.)
		
		Let $A$ be the invertible matrix satisfying $A\alpha^n=\bee_n$ for $1\le n\le k$ and $A\bee_n=\bee_n$ for $k+1\le n\le d.$ Substituting $x=A^Ty,$ the integral is bounded by a constant only depending on $A$, multiplied by
		\begin{equation*}
			\int_{-\infty}^\infty\cdots\int_{-\infty}^\infty\int_{a_{k+1}}^{b_{k+1}}\cdots\int_{a_d}^{b_d}\min_{1\le n\le k}\{M_n e^{y_n},M_{k+1}e^{A\alpha^{k+1}\cdot y}\}dy_d\cdots dy_1.
		\end{equation*}
		Since $A\alpha^{k+1}$ only depends on $y_1,\dots, y_k$, the integrand is independent of the remaining variables. Integrating first over the inner $d-k$ variables, the above equals to $\prod_{i=k+1}^d(b_i-a_i)$ times
		\begin{equation}\label{eq: min int}
			\int_{\R^k}\min_{1\le n\le k}\{M_n e^{y_n},M_{k+1}e^{A\alpha^{k+1}\cdot y}\}dy_k\cdots dy_1.
		\end{equation}
		
		By definition of $A$ and $\alpha^{k+1},$
		\begin{equation}\label{eq: A dot y}
			A\alpha^{k+1}\cdot y = A\left(-\sum_{i=1}^k\frac{\theta_i}{\theta_{k+1}}\alpha^i \right)\cdot y = -\sum_{i=1}^k \frac{\theta_i}{\theta_{k+1}} y_i.
		\end{equation}
		Therefore, if $y^0\in\R^k$ satisfies 
		\begin{equation*}
			e^{y_n^0}= M_n^{-1}\prod_{i=1}^{k+1} M_i^{\theta_i},
		\end{equation*} 
		then by \eqref{eq: A dot y}, $y^0$ also satisfies
		\begin{multline*}
			e^{A\alpha^{k+1}\cdot y^0} = \left(\prod_{\ell=1}^k e^{-\theta_\ell y_\ell^0}\right)^{1/\theta_{k+1}}=\left(\prod_{\ell=1}^k \left(M_\ell^{-1}\prod_{i=1}^{k+1} M_i^{\theta_i}\right)^{-\theta_\ell}\right)^{1/\theta_{k+1}}\\
			= \left(\prod_{\ell=1}^k M_\ell^{\theta_\ell}\right)^{1/\theta_{k+1}}\left(\prod_{i=1}^{k+1} M_i^{\theta_i}\right)^{1-1/\theta_{k+1}}= M_{k+1}^{-1}\prod_{i=1}^{k+1}M_i^{\theta_i}.
		\end{multline*}
		
		Shifting variables by the translation $y\to y+y^0$, the integral \eqref{eq: min int} equals
		\begin{equation*}
			\prod_{i=1}^{k+1}M_i^{\theta_i}\int_{\R^k}\min_{1\le n\le k}\{e^{y_n}, e^{A^T\alpha^{k+1}\cdot y}\}dy_k\cdots dy_1.
		\end{equation*}
		Now $\R^k$ can be written as the union $\{y_1\le 0\}\cup\cdots\cup\{y_k\le 0\}\cup \{\theta_1y_1+\cdots+\theta_ky_k\ge 0\},$ since all $\theta_n>0.$
		The constraint on $\theta_1,\dots,\theta_{k+1}$ guarantees
		\begin{equation*}
			\sup_{\|y\|_2=1}\min_{1\le n\le k}\{y_n, -\sum_{i=1}^k\theta_iy_i/\theta_{k+1}\}<-c<0
		\end{equation*}
		for some positive constant $c$ depending only on the coefficients $\theta_n.$ By \eqref{eq: A dot y} and homogeneity, we establish
		\begin{equation*}
			\min_{1\le n\le k}\{y_n, A\alpha^{k+1}\cdot y\}<-c\|y\|_2.
		\end{equation*} 
		So the minimum inside the integral can be bounded by $e^{-c\|y\|_2}.$ Since the integral of this Gaussian is a uniform constant depending only on $c$ and $k$, we conclude that our initial sum is bounded above by exactly what was required.
	\end{proof}
	In Corollary \ref{cor: linop} below, $\delta$ should be thought of as a Newton distance and $\gamma=2^{d-1}.$
	
	\begin{corollary}\label{cor: linop}
		Let $\alpha^1,\dots, \alpha^k$ be linearly independent vectors in $\N^d$, where $d\ge k.$ Let $\gamma,\delta>0$ and assume $\alpha=\sum_{i=1}^k\theta_i\alpha^i$ for $0<\theta_i<1$ summing to 1. Then
		\begin{multline*}
			\sum_{j_1=0}^{c_1\log(\lambda)}\cdots\sum_{j_d=0}^{c_d\log(\lambda)} \min_{1\le i\le k}\{2^{-\vec{j}\cdot\alpha/\delta},\lambda^{-1/\gamma}2^{\vec{j}\cdot (\alpha^i/\gamma-\alpha/\delta)} \} \\
			\lesssim\begin{cases}
				\lambda^{-1/\delta} \log^{d-k}(\lambda) \text{ if } \delta >\gamma,\\
				\lambda^{-1/\delta} \log^{d}(\lambda) \text{ if } \delta =  \gamma,\\
				\lambda^{-1/\gamma} \text{ if } \delta <\gamma,
			\end{cases}
		\end{multline*}
		where the implicit constant is independent of $\lambda.$
	\end{corollary} 
	
	\begin{proof} There are three cases to consider.
		
		\textbf{Case 1:} Here, we will apply Lemma \ref{lem: lin op general} in the case $\alpha=\delta(1-P+p_1^{-1},\dots, 1-P+p_d^{-1})\in\partial\cn(\phi)$ lies on a codimension $d-k+1$ face, i.e., when there are linearly independent $\alpha^1,\dots, \alpha^k$ such that $\alpha$ lives in the convex hull of $\{\alpha^1,\dots, \alpha^k\}.$ Since $\alpha^1,\dots, \alpha^k$ are linearly independent, without loss of generality we assume $\alpha^1,\dots, \alpha^k, \bee_{k+1},\dots, \bee_d$ are linearly independent in order to apply the above results. 
		
		The inequality $\delta>\gamma$ implies $0<1-\gamma/\delta<1$, so we write $\bz$ as the convex combination
		\begin{equation*}
			\bz= \left(1-\frac{\gamma}{\delta}\right)\left(-\frac{\alpha}{\delta}\right)+\sum_{i=1}^k\frac{\theta_i\gamma}{\delta}\left(\frac{\alpha^i}{\gamma}-\frac{\alpha}{\delta}\right)
		\end{equation*}
		and notice that all coefficients are nonzero. In this case we apply Lemma \ref{lem: lin op general} with $M_n=\lambda^{-1/\gamma}$ for $1\le n\le k$ and $M_{k+1}=1$ to get 
		\begin{equation*}
			\prod_{i=1}^{k+1}M_i^{\theta_i}= \prod_{i=1}^{k}\lambda^{(-1/\gamma)(\theta_i\gamma/\delta)}=\lambda^{-1/\delta}.
		\end{equation*}
		Letting $a_n,b_n$ be as in Lemma \ref{lem: lin op general}, $b_n-a_n=c_n\log(\lambda)$ for $k+1\le n\le d,$ so our final estimate is
		\begin{equation*}
			\sum_{j_1=0}^{c_1\log(\lambda)}\cdots\sum_{j_d=0}^{c_d\log(\lambda)} \min_{1\le n\le k}\{2^{-\vec{j}\cdot\alpha/\delta},\lambda^{-1/\gamma}2^{\vec{j}\cdot (\alpha^n/\gamma-\alpha/\delta)} \}\lesssim \lambda^{-1/\delta}\log^{d-k}(\lambda).
		\end{equation*}
		
		\textbf{Case 2:} If $\delta =\gamma,$ consider
		\begin{equation*}
			\lambda^{-1/\gamma}2^{\vec{j}\cdot (\alpha^n/\gamma-\alpha/\delta)}=\lambda^{-1/\gamma}2^{\vec{j}\cdot (\alpha^n/\gamma-\alpha/\gamma)}:=M_n(\vec{j}).
		\end{equation*} 
		Minimizing only over these $k$ many terms is less than or equal to minimizing over the $\log$-convex combination 
		\begin{equation*}
			\prod_{i=1}^k M_i(\vec{j})^{\theta_i}=\lambda^{-1/\gamma}=\lambda^{-1/\delta}.
		\end{equation*}
		In this case, summing over each variable $j_1,\dots, j_d$ produces an additional log term, granting the upper bound $\lambda^{-1/\delta}\log^d(\lambda).$
		
		\textbf{Case 3:} If $\delta < \gamma,$ then let $M_n(\vec{j})=\lambda^{-1/\gamma}2^{\vec{j}\cdot (\alpha^n/\gamma-\alpha/\delta)}$ as in Case 2. Then
		\begin{equation*}
			\prod_{i=1}^k M_i(\vec{j})^{\theta_i}=\lambda^{-1/\gamma}2^{\vec{j}\cdot(\alpha/\gamma-\alpha/\delta)}=\lambda^{-1/\gamma} \prod_{i=1}^{d} r_i^{-j_i}
		\end{equation*}
		for some positive reals $r_n$, since $1/\gamma < 1/\delta,$ and since all components of $\alpha$ are strictly positive by nondegeneracy of $S.$
		
		Summing over all $j_n$ proves the claim for Case 3: 
		\begin{multline*}
			\sum_{j_1=0}^{c_1\log(\lambda)}\cdots\sum_{j_d=0}^{c_d\log(\lambda)} \min_{1\le n\le k}\{2^{-\vec{j}\cdot\alpha/\delta},\lambda^{-1/\gamma}2^{\vec{j}\cdot (\alpha^n/\gamma-\alpha/\delta)} \}\\
			\lesssim   \sum_{j_1=0}^{\infty}\cdots\sum_{j_d=0}^{\infty} \lambda^{-1/\gamma} \prod_{i=1}^{d} r_i^{-j_i}
			\lesssim \lambda^{-1/\gamma}.
		\end{multline*}
	\end{proof}
	
	\begin{proof}[Proof of Theorem \ref{thm: main theorem}]
		Without loss of generality, it suffices to estimate $\Lambda_d$ only on the orthant $[0,1]^d$. Applying the previously given partition of unity, \eqref{eq: higher size estimate} gives
		\begin{multline*}
			\frac{|\Lambda_d(f_1,\dots,f_d)|}{\prod_{j=1}^d\|f_j\|_{p_j}}\\
			\lesssim\sum_{j_1=1}^\infty\cdots\sum_{j_d=1}^\infty \min_{\alpha\in\cn(S)}\left\lbrace \prod_{k=1}^d \varep_{k}^{1-P+p_k^{-1}}, \lambda^{-2^{1-d}}\prod_{k=1}^d \varep_{k}^{-2^{1-d}\alpha_k+1-P+p_k^{-1}}\right\rbrace.
		\end{multline*}
		
		By Corollary \ref{cor: linop} (taking $\gamma=2^{d-1}$ and $\delta$ as in the statement of Theorem \ref{thm: main theorem}) and our prior reasoning restricting the sum to $j_i\leq c_i\log(\lambda)$, we obtain
		\begin{equation*}
			\frac{|\Lambda_d(f_1,\dots,f_d)|}{\prod_{j=1}^d\|f_j\|_{p_j}}\lesssim
			\begin{cases}
				\lambda^{-1/\delta} \log^{d-k}(\lambda) \text{ if } \delta >2^{d-1},\\
				\lambda^{-1/\delta} \log^{d}(\lambda) \text{ if } \delta =  2^{d-1},\\
				\lambda^{-1/2^{d-1}} \text{ if } \delta <2^{d-1},
			\end{cases}
		\end{equation*}
		which becomes the conclusion of Theorem \ref{thm: main theorem} by merging the $\delta=2^{d-1}$ case with the $\delta\ge 2^{d-1}$ case.

	\end{proof}
	
	The proofs of Lemma \ref{lem: lin op general} and Corollary \ref{cor: linop} could also be used to expand the class of phases considered in \cite{GiGrXi18}. In particular, their results could be extended to nondegenerate phases with Newton distance less than or equal to 2, and in this case the decay obtained is no worse than $\lambda^{-1/2}\log^d(\lambda)$.
	
	Although we cannot interpolate the final estimate for $\Lambda_d$ with other local Loomis-Whitney inequalities to obtain a result like Proposition \ref{prop: interpolation sharpness}, one could interpolate local estimates and then sum as in  Lemma \ref{lem: lin op general} and Corollary \ref{cor: linop} to obtain sharp results for other $p_j$ smaller than in Theorem \ref{thm: main theorem}. For example, interpolating local estimates for $(8,4,4,2)$ with $(p(4),\dots,p(4))=(24/7,\dots,24/7)$ with $\theta_1=\theta_2=1/2$ provides the following local bound for $(48/10,48/13, 48/13, 48/19):=(p_1,\dots,p_4)$:
	\begin{equation*}
		|\Lambda_d(f_1,\dots, f_d)| \lesssim (\lambda\varep^{\alpha})^{-2^{-4}} \prod_{j=1}^{d} \varep_j^{1-P+p_j^{-1}}\|f_j\|_{p_j}
	\end{equation*}
	Summing with Corollary \ref{cor: linop} grants sharp decay for these exponents. Transforming this example into a precise theorem leads to interpolations that are quite tedious to analyze compared to Proposition \ref{prop: interpolation sharpness} and its proof, so we leave it for the interested reader to work out any details, should the need arise.

	\Addresses
\end{document}